\documentclass{amsart}
\usepackage{amssymb}
\usepackage{epsfig}

\title{Combinatorial Constructions of Weight Bases: The Gelfand-Tsetlin Basis}

\author{Patricia Hersh and Cristian Lenart}

\address{Department of Mathematics, Rawles Hall, 
Indiana University, Bloomington, IN 47405}
\email{phersh@indiana.edu}

\address{Department of Mathematics and Statistics, State University of New York at Albany, Albany, NY 12222}
\email{lenart@albany.edu}

\thanks{P. H.  was partially supported by National Science Foundation 
grant DMS-0500638.  C. L. was partially supported by National Science Foundation 
grants DMS-0403029 and DMS-0701044}

\numberwithin{equation}{section}

\theoremstyle{plain}
\newtheorem{theorem}{Theorem}[section]
\newtheorem{proposition}[theorem]{Proposition}
\newtheorem{lemma}[theorem]{Lemma}

\theoremstyle{definition}

\theoremstyle{remark}
\newtheorem{remark}[theorem]{Remark}

\newcommand{\ts}{\,}
\newcommand{\gl}{\mathfrak{gl}}

\def\w{{\rm wt}}

\newcommand{\stackprod}[2]{\prod_{\begin{array}{c}\vspace{-6mm}\;\\ \vspace{-1mm}\scriptstyle{#1}\\ \scriptstyle{#2}\end{array}} }

\begin{document}
\bibliographystyle{plain}

\begin{abstract} This work is part of a project on weight bases for the irreducible representations of semisimple Lie algebras with respect to which the representation matrices of the Chevalley generators are given by explicit formulas. In the case of $\mathfrak{sl}_n$, the celebrated Gelfand-Tsetlin basis is the only such basis known. Using the setup of supporting graphs developed by Donnelly, we present a simple combinatorial proof of the Gelfand-Tsetlin formulas based on a rational function identity. Some properties of the Gelfand-Tsetlin basis are derived via an algorithm for solving certain equations on the lattice of semistandard Young tableaux. 
\end{abstract}

\maketitle

\section{Introduction}

This work is related to combinatorial constructions of weight bases for the irreducible representations of semisimple Lie algebras on which the action of the Chevalley generators is made explicit. We will use the setup introduced by Donnelly \cite{donecr,donecf,donepb}. The main idea is to encode a weight basis into an edge-colored ranked poset (called a supporting graph), whose Hasse diagram has its edges labeled with two complex coefficients. This structure is known as a representation diagram, and it explicitly gives the action of the Chevalley generators of the Lie algebra on the weight basis. Verifying that an assignment of labels to an edge-colored poset is a representation diagram amounts to checking that the labels satisfy some simple relations. Thus, constructing a basis of a representation amounts to solving a system of equations associated to a poset. 

The goal in the basis construction is finding supporting graphs with a small number of edges, possibly edge-minimal ones (with respect to inclusion); this amounts to finding a basis for which the action of the Chevalley generators is expressed by a formula with a small number of terms. It is often the case that the labels of an edge-minimal supporting graph are essentially the unique solution of the corresponding system of equations. This property is known as the solitary property of the associated basis. Another interesting property of many supporting graphs constructed so far is that Kashiwara's crystal graphs of the corresponding representations \cite{kascqa,kascbq} are subgraphs. Thus, the theory of supporting graphs can be viewed as an extension of the theory of crystal graphs, which has attracted considerable interest in the combinatorics community in recent years. Finally, many supporting graphs are better behaved as posets than the corresponding crystal graphs/posets, being lattices, modular lattices, or even distributive lattices. 

In the case of irreducible representations of $\mathfrak{sl}_n$, the celebrated Gelfand-Tsetlin basis is the only known basis with respect to which the representation matrices of the Chevalley generators are given by explicit formulas. It turns out that the supporting graph of the Gelfand-Tsetlin basis is edge-minimal, solitary, and a distributive lattice \cite{donepb}; we will call it the Gelfand-Tsetlin lattice. Donnelly constructed solitary, edge-minimal, and modular lattice supporting graphs for certain special representations, most notably: the fundamental representations of $\mathfrak{sp}_{2n}$ and $\mathfrak{so}_{2n+1}$ \cite{donecr,donsad,donecf}, the ``one-rowed'' representations of  $\mathfrak{so}_{2n+1}$ \cite{dlpcro}, and the adjoint representations of all simple Lie algebras \cite{doneba}. Molev constructed bases of Gelfand-Tsetlin type (i.e., which are compatible with restriction to the Lie subalgebras of lower rank) for all irreducible representations of the symplectic and orthogonal Lie algebras \cite{molbrs,molwbg,molwbr}. The corresponding representation diagrams (i.e., the action of a system of Chevalley generators on the basis) are not explicitly given, but they can be derived from Molev's formulas for the action of certain elements spanning the Lie algebra. As posets, these supporting graphs are not lattices in general, and there are indications that they are not edge-minimal in general, either. 

Our ultimate goal is finding edge-minimal supporting graphs for symplectic and orthogonal representations, as well as studying their combinatorics. As a first step, in this paper we revisit the Gelfand-Tsetlin basis for $\mathfrak{sl}_n$, by studying it in Donnelly's combinatorial setup. This allows us to show that the construction of the Gelfand-Tsetlin basis relies on nothing more than a simple rational function identity. Moreover, the corresponding solitary and edge-minimal properties, which were derived via theoretical considerations in \cite{donepb}, are proved here in a very explicit way, by a simple algorithm for solving equations on the Gelfand-Tsetlin lattice. We envision that such algorithms and rational function identities will play a crucial role in our future work. On the other hand, let us note that the proofs of the Gelfand-Tsetlin formulas that appeared since the original paper \cite{gatfdr} by Gelfand and Tsetlin in the fifties (which contained no proof) use more sophisticated algebraic methods, based on: lowering operators \cite{hpyobi,zhecgs,zhecgr}, boson-calculus techniques \cite{babors}, polynomial expressions for Wigner coefficients \cite{gouome}, the theory of the Mickelsson algebras \cite{zheits}, and the quantum algebras called Yangians \cite{molgtb,natygt}. In turn, Molev's constructions \cite{molbrs,molwbg,molwbr} of his bases for orthogonal and symplectic representations are based on complex calculations related to Yangians. 

In terms of the combinatorial model for describing the supporting graph, we use semistandard Young tableaux rather than Gelfand-Tsetlin patterns. By analogy, we expect to use Kashiwara-Nakashima or De Concini tableaux \cite{decsst,kancgr} for the representations of the symplectic and orthogonal algebras. Note that these tableaux were already used in Donnelly's work mentioned above, whereas Molev's work is based on Gelfand-Tsetlin patterns of type $B-D$. 

\medskip

{\bf Acknowledgement.} We are grateful to Robert Donnelly for explaining to us his work on supporting graphs for representations of semisimple Lie algebras.

\section{Background}

\subsection{Supporting graphs}\label{suppgraphs}

We follow \cite{donepb,dlpsem} in describing the setup of supporting graphs/representation diagrams. We consider finite ranked posets, and we identify a poset with its Hasse diagram, thus viewing it as a directed graph with edges $s\rightarrow t$ for each covering relation 
$s\lessdot t$. These edges will be colored by a set $I$, and we write $s\stackrel{i}{\rightarrow}t$ to indicate that the corresponding edge has color $i \in I$. 
The connected components of the subgraph with edges colored $i$ are called {\em $i$-components}. Besides a given color, each edge $s\rightarrow t$ is labeled with two complex coefficients, which are not both 0, and which are denoted by $c_{t,s}$ and $d_{s,t}$. Given the poset $P$, let $V[P]$ be the complex vector space with basis $\{v_s\}_{s\in P}$. We define operators $X_i$ and $Y_i$ on $V[P]$ for $i$ in $I$, as follows:
\begin{equation}\label{updown}X_i\,v_s:=\sum_{t\::\:s\stackrel{i}{\rightarrow}t}c_{t,s}v_t\,,\;\;\;\;\;\;\;\;\;\;\;Y_i\,v_t:=\sum_{s\::\:s\stackrel{i}{\rightarrow}t}d_{s,t}v_s\,.\end{equation}
For each vertex $s$ of $P$, we also define a set of integers $\{m_i(s)\}_{i\in I}$ by $m_i(s):=2\rho_i(s)-l_i(s)$, where $l_i(s)$ is the rank of the $i$-component containing $s$, and $\rho_i(s)$ is the rank of $s$ within that component. 

Let $\mathfrak g$ be a semisimple Lie algebra with {\em Chevalley generators} $\{X_i,\,Y_i,\,H_i\}_{i\in I}$. Let $\{\omega_i\}_{i\in I}$ and $\{\alpha_i\}_{i\in I}$ denote the {\em fundamental weights} and {\em simple roots} of the corresponding root system, respectively. Consider an edge-colored and edge-labeled ranked poset $P$, as described above. Let us assign a weight to each vertex by $\w(s):=\sum_{i\in I}m_i(s)\omega_i$. We say that the edge-colored poset $P$ satisfies the {\em structure condition} for $\mathfrak g$ if $\w(s)+\alpha_i=\w(t)$ whenever $s\stackrel{i}{\rightarrow}t$. 

We now define two conditions on the pairs of edge labels $(c_{t,s},d_{s,t})$. We call $\pi_{s,t}:=c_{t,s}\,d_{s,t}$ an {\em edge product}. The edge-labeled poset $P$ satisfies the {\em crossing condition} if for any vertex $s$ and any color $i$ we have
\begin{equation}\label{crossing}\sum_{r\::\:r\stackrel{i}{\rightarrow}s}\pi_{r,s}-\sum_{t\::\:s\stackrel{i}{\rightarrow}t}\pi_{s,t}=m_i(s)\,.\end{equation}
A relation of the above form is called a {\em crossing relation}. The edge-labeled poset $P$ satisfies the {\em diamond condition} if for any pair of vertices $(s,t)$ of the same rank and any pair of colors $(i,j)$, possibly $i=j$, we have
\begin{equation}\label{diamond}\sum_{u\::\:s\stackrel{j}{\rightarrow}u\;{\rm and}\;t\stackrel{i}{\rightarrow}u}c_{u,s}\,d_{t,u}=\sum_{r\::\:r\stackrel{i}{\rightarrow}s\;{\rm and}\;r\stackrel{j}{\rightarrow}t}d_{r,s}\,c_{t,r}\,,\end{equation}
where an empty sum is zero. If for given pairs $(s,t)$ and $(i,j)$ there is a unique vertex $u$ such that $s\stackrel{j}{\rightarrow}u$ and $t\stackrel{i}{\rightarrow}u$, as well as a unique vertex $r$ such that $r\stackrel{i}{\rightarrow}s$ and $r\stackrel{j}{\rightarrow}t$, then the relation (\ref{diamond}) for these pairs and for the reverse pairs  $(t,s)$ and $(j,i)$ reduce to
\begin{equation}\label{diamondred}
c_{u,s}\,d_{t,u}=d_{r,s}\,c_{t,r}\,,\;\;\;\;\;\;\;\;\;\;\;c_{u,t}\,d_{s,u}=d_{r,t}\,c_{s,r}\,;
\end{equation}
these relations imply 
\begin{equation}\label{diam1}
\pi_{s,u}\,\pi_{t,u}=\pi_{r,s}\,\pi_{r,t}\,.
\end{equation}
A relation of the form (\ref{diamond}), (\ref{diamondred}), or (\ref{diam1}) is called a {\em diamond relation}. 

We want to define a representation of $\mathfrak g$ on $V[P]$ by letting the Chevalley generators $X_i$ and $Y_i$ act as in (\ref{updown}), and by setting 
\begin{equation}\label{acth}H_i\,v_s:=m_i(s)\,v_s\,.\end{equation}
The following proposition gives a necessary and sufficient condition on the edge labels.

\begin{proposition}\label{cond}\cite[Lemma 3.1]{dlpsem}\cite[Proposition 3.4]{donepb}
Given an edge-colored and edge-labeled ranked poset $P$, the actions {\rm (\ref{updown})} and {\rm (\ref{acth})} define a representation of $\mathfrak g$ on $V[P]$ if and only if $P$ satisfies the diamond, crossing, and structure conditions. 
\end{proposition}

If the given conditions hold, the set $\{v_s\}_{s\in P}$ is a {\em weight basis} of the given representation, the edge-colored poset $P$ is called a {\em supporting graph} of the representation, and $P$ together with its edge-labels is called a {\em representation diagram}. Some general properties of supporting graphs were derived in [Section 3]\cite{donepb}. 

Many supporting graphs constructed so far have special properties, which we mention below. A supporting graph is called {\em edge-minimal} if no proper subgraph of it is the supporting graph for a weight basis of the corresponding representation. Two weight bases related by a diagonal transition matrix are called {\em diagonally equivalent}. The supporting graph for a weight basis is called {\em solitary} if the diagonally equivalent bases are the only ones with the same supporting graph. Hence, up to diagonal equivalence, a solitary weight basis is uniquely determined by its supporting graph. Two representation diagrams are called {\em edge product similar} if there is a poset isomorphims between them which preserves the edge colors and the edge products. The following lemma highlights the importance of edge products.

\begin{lemma}\label{edgeprod}\cite[Lemma 4.2]{dlpsem} Let $L$ be a representation diagram for a weight basis $\mathcal B$ of $V$ which is connected (as a graph) and modular (as a poset). A representation diagram $K$ which is edge product similar to $L$ is the representation diagram of a diagonally equivalent basis to $\mathcal B$.  
\end{lemma}

\subsection{The Gelfand-Tsetlin basis} 

Let $E_{ij}$, $i,j=1,\dots,n$ denote the standard basis of the general
linear Lie algebra
$\gl_n$ over the field of complex numbers.

Consider a partition $\lambda$ with at most $n$ rows, that is a weakly decreasing sequence of integers $(\lambda_1\ge\lambda_2\ge\dots\ge\lambda_n\ge 0)$. Let $V(\lambda)$ be the  
finite-dimensional irreducible representation of $\gl_n$ with highest weight $\lambda$. 
A basis of $V(\lambda)$ is parametrized by {\it Gelfand--Tsetlin patterns\/} $\Lambda$ associated with
$\lambda$; these are arrays of integer row vectors
\begin{align}\label{gtpatt}
&\qquad\lambda_{n1}\qquad\lambda_{n2}
\qquad\qquad\cdots\qquad\qquad\lambda_{nn}\nonumber\\
&\qquad\qquad\lambda_{n-1,1}\qquad\ \ \cdots\ \ 
\ \ \qquad\lambda_{n-1,n-1}\nonumber\\
&\quad\qquad\qquad\cdots\qquad\cdots\qquad\cdots\\
&\quad\qquad\qquad\qquad\lambda_{21}\qquad\lambda_{22}\nonumber\\
&\quad\qquad\qquad\qquad\qquad\lambda_{11}  \nonumber
\end{align}
such that the upper row coincides with $\lambda$ and 
the following conditions hold:
\begin{equation}\label{aconl}
\lambda_{ki}\ge\lambda_{k-1,i}\,,\qquad  
\lambda_{k-1,i}\ge \lambda_{k,i+1}\,,\qquad 
i=1,\dots,k-1
\end{equation}
for each $k=2,\dots,n$. Let us set
$l_{ki}=\lambda_{ki}-i+1$.

\begin{theorem}\label{gtbasis}\cite{gatfdr} There exists a basis 
$\{\xi_{\Lambda}\}$ of $V(\lambda)$ parametrized by the corresponding 
patterns $\Lambda$ such that the action
of generators of $\gl_n$ is given by the following formulas:
\begin{align}\label{ekk}
E_{kk}\ts \xi_{\Lambda}&=\left(\sum_{i=1}^k\lambda_{ki}
-\sum_{i=1}^{k-1}\lambda_{k-1,i}\right)
\xi_{\Lambda},
\\
\label{ekk+1}
E_{k,k+1}\ts \xi_{\Lambda}
&=-\sum_{i=1}^k \frac{(l_{ki}-l_{k+1,1})\cdots (l_{ki}-l_{k+1,k+1})}
{(l_{ki}-l_{k1})\cdots \wedge\cdots(l_{ki}-l_{kk})}
\ts
\xi_{\Lambda+\delta_{ki}},\\
\label{ek+1k}
E_{k+1,k}\ts \xi_{\Lambda}
&=\sum_{i=1}^k \frac{(l_{ki}-l_{k-1,1})\cdots (l_{ki}-l_{k-1,k-1})}
{(l_{ki}-l_{k1})\cdots \wedge\cdots(l_{ki}-l_{kk})}
\ts
\xi_{\Lambda-\delta_{ki}}.
\end{align}
The arrays $\Lambda\pm\delta_{ki}$
are obtained from $\Lambda$ by replacing $\lambda_{ki}$
by $\lambda_{ki}\pm1$. It is supposed
that $\xi_{\Lambda}=0$ if the array $\Lambda$ is not a pattern;
the symbol $\wedge$ indicates that the zero factor in the denominator
is skipped.
\end{theorem}

\section{The representation diagram of the Gelfand-Tsetlin basis} \label{repgt}

We restrict ourselves to $\mathfrak{sl}_n$, for which we have the standard choice of Chevalley generators $H_k:=E_{k,k}-E_{k+1,k+1}$, $X_k:=E_{k,k+1}$, and $Y_k:=E_{k+1,k}$, for $k$ in $I:=[n-1]=\{1,\ldots,n-1\}$.

We identify partitions with Young diagrams, so we refer to the cells $(i,j)$ of a partition $\lambda$. 
There is a natural
bijection between the Gelfand-Tsetlin patterns associated with $\lambda$ and
{\em semistandard Young tableaux} (SSYT) of shape $\lambda$ with entries in $[n]$, see e.g. \cite{macsfh}.
A pattern $\Lambda$ can be viewed as a sequence of partitions
\begin{equation}\label{aseq}
\lambda^{(1)}\subseteq \lambda^{(2)}\subseteq\cdots 
\subseteq \lambda^{(n)}=\lambda, \nonumber
\end{equation}
with 
$\lambda^{(k)}=(\lambda_{k1},\dots,\lambda_{kk})$. We let $\lambda^{(0)}$ be the empty partition. 
Conditions \eqref{aconl} mean that the skew diagram
$\lambda^{(k)}/\lambda^{(k-1)}$ is a horizontal strip. The SSYT $T$ associated with $\Lambda$ is then obtained by filling the cells in $\lambda^{(k)}/\lambda^{(k-1)}$ with the entry $k$, for each $k=1,\ldots,n$. 

We now define a representation diagram for $\mathfrak{sl}_n$ with edge colors $I$ on the SSYT of shape $\lambda$ with entries in $[n]$. We have an edge $S\stackrel{k}{\rightarrow}T$ whenever the tableau $T$ is obtained from $S$ by changing a single entry $k+1$ into $k$; necessarily, this is the leftmost entry $k+1$ in a row. The corresponding poset, which is known to be a distributive lattice, will be called the {\em Gelfand-Tsetlin lattice}, and will be denoted by $GT(\lambda)$. 

 To define the edge labels on $GT(\lambda)$, fix a SSYT $T$ in this lattice, and let the corresponding Gelfand-Tsetlin pattern $\Lambda$ be denoted as in (\ref{gtpatt}). The labels on the incoming/outgoing edges to/from $T$ which are colored $k$ will only depend on the corresponding partitions $\lambda^{(k-1)}$, $\lambda^{(k)}$, and $\lambda^{(k+1)}$. The {\em outer rim} $R$ of $\lambda^{(k)}$ consists of all cells $(i,j)$ not in $\lambda^{(k)}$ such that at least one of the cells $(i,j-1),\,(i-1,j),\,(i-1,j-1)$ belongs to $\lambda^{(k)}$. A cell $(i,j)$ of $R$ is called an {\em outer corner} if $(i,j-1)$ and $(i-1,j)$ belong to $R$, and an {\em inner corner} if neither $(i,j-1)$ nor $(i-1,j)$ belongs to $R$. The inner and outer corners are interleaved, and the number of the former exceeds by 1 the number of the latter.  Number the cells of $R$ from northeast to southwest starting from 1. Let $a_1<\ldots<a_{p}$ be the numbers attached to the inner corners, and $a_1'<\ldots<a_{p-1}'$ the numbers attached to the outer corners. Furthermore, if $r_1=1<r_2<\ldots<r_{p}$ are the rows of the inner corners, we denote by $b_i$ the length (which might be 0) of the component of $\lambda^{(k+1)}/\lambda^{(k)}$ in row $r_i$, for $i=1,\ldots,p$; similarly, we denote by $b_{i}'$ the length of the component of $\lambda^{(k)}/\lambda^{(k-1)}$ in row $r_{i+1}-1$, for $i=1,\ldots,p-1$. Note that  row $r_i$ of $T$ contains both $k$ and $k+1$ precisely when $r_{i+1}-1=r_i$, $b_i>0$, and $b_i'>0$. The notation introduced above is illustrated in the figure below; the Young diagram with a bold boundary is that of $\lambda^{(k)}$, while the indicated cells are those in $\lambda^{(k+1)}/\lambda^{(k)}$ and $\lambda^{(k)}/\lambda^{(k-1)}$.

\vspace{6mm}


\begin{figure}[ht]
\mbox{\epsfig{file=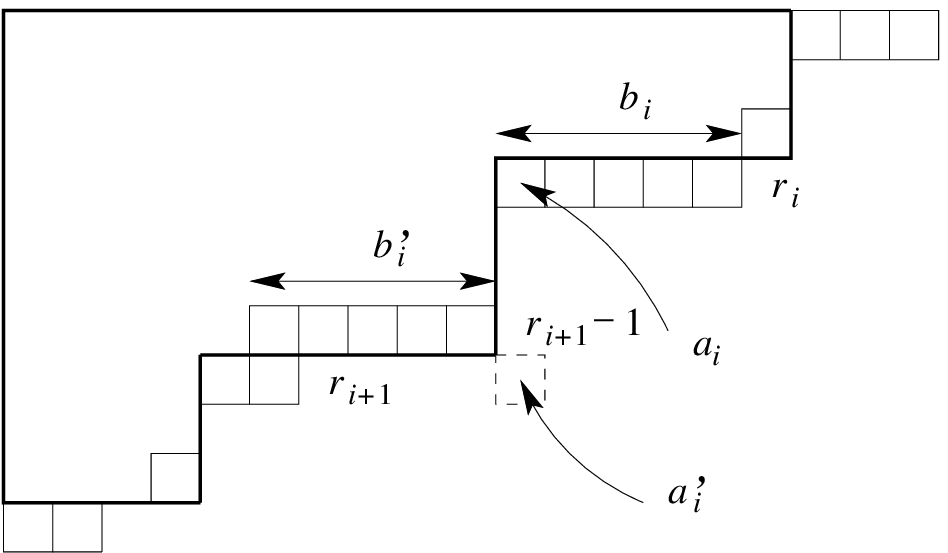}}
\label{fig}
\end{figure}
 
\vspace{6mm}

Assume that we have $T\stackrel{k}{\rightarrow}U$, and that the entry $k+1$ in $T$ changed into $k$ is in some row $r_i$ for $1\le i\le p$ (this is always the case). Thus, the Gelfand-Tsetlin pattern corresponding to $U$ is $\Lambda+\delta_{kr_i}$. Then let
\begin{equation}\label{cst}
c_{U,T}:=b_i \prod_{j=1}^{i-1}\left(1+\frac{b_j}{a_i-a_j}\right)\prod_{j=i+1}^{p}\left(1-\frac{b_j}{a_j-a_i}\right)\,.
\end{equation}
Similarly, assume that we have $S\stackrel{k}{\rightarrow}T$, and that the entry $k$ in $T$ changed into $k+1$ is in some row $r_{i+1}-1$ for $1\le i\le p-1$ (this is always the case). Thus, the Gelfand-Tsetlin pattern corresponding to $S$ is $\Lambda-\delta_{k,r_{i+1}-1}$. Then let
\begin{equation}\label{dst}
d_{S,T}:=b_i' \prod_{j=1}^{i-1}\left(1-\frac{b_j'}{a_i'-a_j'}\right)\prod_{j=i+1}^{p-1}\left(1+\frac{b_j'}{a_j'-a_i'}\right)\,.
\end{equation}
The horizontal strip conditions on $\lambda^{(k+1)}/\lambda^{(k)}$ and $\lambda^{(k)}/\lambda^{(k-1)}$ guarantee that $c_{U,T}$ and $d_{S,T}$ are nonnegative rational numbers which only become $0$ when $b_i=0$, respectively $b_i'=0$.

It is not hard to calculate the edge products using (\ref{cst}) and (\ref{dst}). For instance, we have
\begin{equation}\label{pist}
\pi_{T,U}=b_i \prod_{j=1}^{i-1}\left(1-\frac{b_j}{a_j-a_i}\right)\prod_{j=i+1}^{p}\left(1-\frac{b_j}{a_j-a_i}\right)\prod_{j=1}^{i-1}\left(1+\frac{b_j'}{a_j'-a_i}\right)\prod_{j=i}^{p-1}\left(1+\frac{b_j'}{a_j'-a_i}\right)\,,
\end{equation}
and a similar formula for $\pi_{S,T}$.

\begin{proposition}\label{eps}
The edge-colored and edge-labeled poset defined above is edge product similar, as a representation diagram, to that of the Gelfand-Tsetlin basis. 
\end{proposition}

\begin{proof}
Observe first that, if $j<i$ and $\lambda_j^{(p)}>\lambda_i^{(m)}$, then the difference $l_{pj}-l_{mi}$ is the length of the hook  with rightmost cell $(j,\lambda_j^{(p)})$ and bottom cell $(i,\lambda_i^{(m)}+1)$. Still assuming $i>j$, this means that  $l_{kr_j}-l_{kr_i}=a_i-a_j$ and $l_{k+1,r_j}-l_{kr_i}=a_i-a_j+b_j$; thus, by pairing these two factors, we obtain
\[\frac{l_{kr_i}-l_{k+1,r_j}}{l_{kr_i}-l_{kr_j}}=1+\frac{b_j}{a_i-a_j}\,.\]
The other three types of brackets in (\ref{cst}) and (\ref{dst}) are obtained in a similar way from the Gelfand-Tsetlin formulas (\ref{ekk+1}) and (\ref{ek+1k}). However, there are exceptions when $r_p=k+1$ (see below), namely the bracket corresponding to $j=p$ in (\ref{cst}), and the bracket corresponding to $j=p-1$ in (\ref{dst}). Finally, note that the quotient in (\ref{ekk+1}) corresponding to the rows different from $r_j$ is 1; a similar statement holds for (\ref{ek+1k}). 

Let us now discuss the exception mentioned above. Assume that $r_p=k+1$.  The coefficient $c_{U,T}$ differs from the corresponding coefficient in (\ref{ekk+1}) by a factor $1/(a_p-a_i)$. Similarly, the coefficient $d_{S,T}$ differs from the corresponding coefficient in (\ref{ek+1k}) by a factor $a_{p-1}'-a_i'+b_{p-1}'=a_p-a_i'$. The last equality holds because the first $\lambda_{k}^{(k)}$ entries in row $k$ of $T$ are equal to $k$.

It is clear from (\ref{pist}) and the discussion related to the slight discrepancy between (\ref{ekk+1})-(\ref{ek+1k}) and (\ref{cst})-(\ref{dst}) that the representation diagram of the Gelfand-Tsetlin basis is edge product similar to the one described in this section. 
\end{proof}

\section{A proof of the Gelfand-Tsetlin formulas}

Based on Proposition \ref{cond}, the proof of the Gelfand-Tsetlin formulas amounts to the first statement in the theorem below.

\begin{theorem}\label{relations} The edge-colored and edge-labeled poset defined in Section {\rm \ref{repgt}} satisfies the diamond, crossing, and structure conditions. Hence, it is the representation diagram of the irreducible representation of $\mathfrak{sl}_n$ with highest weight $\lambda$. 
\end{theorem}

\begin{proof}
We use the notation in Section \ref{repgt} related to the SSYT $T$, as well as the notation in Section \ref{suppgraphs}. Let $\mu_k:=b_1'+\ldots+b_{p-1}'$ be the number of entries $k$ in $T$. We have $m_k(T)=\mu_k-\mu_{k+1}$ by definition. Furthermore, $\w(T)=\mu_1\varepsilon_1+\ldots+\mu_n\varepsilon_n$ since, by definition, $\langle\w(T),\alpha_k\rangle=m_k(T)$ (recall that $\alpha_k=\varepsilon_k-\varepsilon_{k+1}$ is the dual basis element to $\omega_k$ under the scalar product $\langle\varepsilon_i,\varepsilon_j\rangle=\delta_{ij}$). Hence, given $T\stackrel{k}{\rightarrow}U$, the structure condition $\w(T)+\alpha_k=\w(U)$ is verified. 

We now address the crossing condition. We have $m_k(T)=b_1'+\ldots+b_{p-1}'-b_1-\ldots-b_p$. In order to simplify formulas, we make the substitution $x_{2i-1}:=a_i$, $y_{2i-1}=-b_i$ for $i=1,\ldots,p$, and $x_{2i}:=a_i'$, $y_{2i}:=b_i'$ for $i=1,\ldots,p-1$. Based on (\ref{pist}), the crossing relation (\ref{crossing}) for vertex $T$ and color $k$ can be written as follows:
\begin{equation}\label{ratfn}
\sum_{i=1}^N y_i\stackprod{1\le j\le N}{j\ne i}\left(1+\frac{y_j}{x_j-x_i}\right)=\sum_{i=1}^N y_i\,,
\end{equation}
where $N:=2p-1$. Note that we can take the above sum over all $i$ from 1 to $2p-1$ because when we attempt an illegal change of a $k+1$ into $k$ or viceversa, the corresponding term is 0. Indeed, assume for instance that we intend to change a $k+1$ in row $r_i$, but this is illegal because we have a $k$ immediately above it. This means that $a_i-a_{i-1}'=b_{i-1}'$, and therefore the right-hand side of (\ref{pist}) cancels. 

Now let us prove the rational function identity (\ref{ratfn}). It is not hard to see that its left-hand side is invariant under the diagonal action of the symmetric group $S_N$ on the variables $x_1,\ldots,x_N$ and $y_1,\ldots,y_N$. Let us expand and extract the coefficient of $y_1y_2\ldots y_k$. For $k=1$ it is clearly 1. For $2\le k\le N$, it is the following symmetric rational function in $x_1,\ldots,x_k$:
\[\sum_{i=1}^k\stackprod{1\le j\le k}{j\ne i}\frac{1}{x_j-x_i}\,.\]
The common denominator is the Vandermonde determinant in $x_1,\ldots,x_k$. Thus the numerator is an antisymmetric polynomial in the same variables, but its degree is $\binom{k}{2}-(k-1)=\binom{k-1}{2}$, so it has to be 0. This concludes the proof of (\ref{ratfn}) by the symmetry of its left-hand side under the diagonal action of $S_N$. 

Since the Gelfand-Tsetlin lattice $GT(\lambda)$ is a distributive lattice, the diamond relations take the simpler form (\ref{diamondred}). Assume that we have $T\stackrel{k}{\rightarrow}U$, $S\stackrel{l}{\rightarrow}U$, $R\stackrel{l}{\rightarrow}T$, and $R\stackrel{k}{\rightarrow}S$. We use the same parameters $a_i,b_i,a_i',r_i$ for the SSYT $T$ as in Section \ref{repgt}. We need to show that $c_{U,T}\,d_{S,U}=d_{R,T}\,c_{S,R}$. This is trivial except for $l=k$ and $l=k+1$, because then $c_{U,T}=c_{S,R}$ and $d_{S,U}=d_{R,T}$. 

Now consider the case $l=k$. Let $r_i$ be the row containing the changed $k+1$ in $T$ and $R$ (for obtaining $S$), and let $r_{j+1}-1$ be the row containing the changed $l=k$ in $U$ (for obtaining $S$) and $T$. Assume that $j\ge i$, the other case being completely similar. Then, by (\ref{cst}) and (\ref{dst}), we have
\[\frac{d_{S,U}}{d_{R,T}}=1-\frac{1}{a_j'-a_i}=\frac{c_{S,R}}{c_{U,T}}\,.\]

Finally, consider the case $l=k+1$. Let $r_i$ be the row containing the changed $k+1$ in $T$ (for obtaining $U$) and $R$, and let $r_{j}$ be the row containing the changed $l=k+1$ in $U$ and $T$ (for obtaining $R$). Assume that $j>i$, the other case being completely similar. Then, by (\ref{cst}) and (\ref{dst}), we have
\begin{align*}&\frac{c_{U,T}}{c_{S,R}}=\left(1-\frac{b_j}{a_j-a_i}\right)\left(1-\frac{b_j-1}{a_j-a_i}\right)^{-1}\,,\\
& \frac{d_{R,T}}{d_{S,U}}=\left(1-\frac{b_i}{a_j-b_j-a_i+b_i}\right)\left(1-\frac{b_i-1}{a_j-b_j-a_i+b_i}\right)^{-1}\,.\end{align*}
A simple calculation shows that the two quotients are equal. 

By Proposition \ref{cond}, the edge-colored and edge-labeled poset defined in Section {\rm \ref{repgt}} encodes a representation of $\mathfrak{sl}_n$. Consider the basis vector corresponding to the maximum of the poset $GT(\lambda)$, namely to the SSYT of shape $\lambda$ with all entries in row $i$ equal to $i$. This is clearly a highest weight vector, and relation (\ref{acth}) shows that the given representation is the irreducible one with highest weight $\lambda$. 
\end{proof}

\section{Edge-minimality and the solitary property}

In this section, we give an algorithm for determining the edge products in a Gelfand-Tsetlin 
lattice from the vertex weights $\{ m_i(s) \} $.  
The idea is to show first how the edge products of color 1 are forced
by the vertex weights; then as an inductive step, 
we show that after edge products have been determined for 
all edges  colored $1,2,\dots ,k-1$, then there
is an algorithm forcing the edge products for all of the edges colored $k$.  

\begin{remark}\label{diamond-product}
If the edge products for three of the four covering relations comprising a diamond are known, then
the diamond relation (\ref{diam1}) will determine the fourth.
\end{remark}

We begin by describing the  algorithm for the edges
colored 1. 
Notice that the $1$-components are chains, or in  other words for each
poset element $T$ there is at most one $U$ covering $T$ such that the covering relation $T\prec U$
is colored
1 and likewise 
there is also at most one $R$ covered by $T$  such that the
covering relation $R\prec T$  is colored 1.
This is immediate from the fact that the value 1 may only appear in the first row of 
a semistandard Young tableau, so that any semistandard Young tableau has at most one copy of 
the value 1 which may be incremented to a 2 still yielding a semistandard Young tableau, since the
only candidate is the rightmost 1 in the first row.

For any minimal or maximal element $T$ in a $1$-component, the crossing condition
forces the edge product for the unique covering involving $T$.  We may keep 
repeating this idea, i.e. at each step using the crossing condition to force one additional
edge product, namely one  at a poset
element $S$ such that all other edge products for covering relations
involving $S$ that are colored 1 have already had
their values forced.  In this manner, we determine all edge products for covering relations colored 1.

Before turning to the inductive step, we make a few 
observations about semistandard Young tableaux that will be used in our upcoming algorithm
for determining edge products of color $k$ once the edge products 
are known for colors $1,2,\dots ,k-1$.  
We now think of an element $T$ of the Gelfand-Tsetlin lattice as a
semistandard Young tableau, using the earlier description of covering relations $T\prec U$ as
pairs of SSYT in which $U$ is obtained from $T$ by incrementing an entry from $i$ to $i+1$.
Recall that this covering relation $T\prec U$  is then said to have color $i$.  

\begin{remark}
The value $k$ may only appear within the first $k$ rows of a SSYT, since the column 
immediately above any occurance of the value $k$ must consist of a strictly increasing 
sequence of positive integers.  
Consequently, 
there are at most $k$ upward edges and at most $k$ 
downward edges colored $k$ originating at any specified element $T$ in the 
Gelfand-Tsetlin lattice, since only the rightmost $k$ in a row may yield an upward covering
relation colored $k$ and only the leftmost $k$ may yield a downward one.
\end{remark}

Let us call a value $k+1$ in a SSYT {\it decrementable} if it may be replaced by $k$ with the 
result still being a SSYT.  Similarly, call  a value 
$k$ in a SSYT {\it incrementable} if it may be replaced by the value $k+1$ to yield a SSYT.

\begin{remark}
Let $d(k)$ be the number of rows in a SSYT which contain a decrementable
copy  of the value $k+1$.  Then 
$d(k)$  is at most one more than 
the number of rows containing a value strictly smaller than $k$ 
which is incrementable. 
This is because each decrementable $k+1$ which is not in the first row has a value $i<k$ immediately above it; this
$i$ may be incremented to obtain a SSYT
unless it has another $i$ to its immediate right, but in that case the rightmost 
$i$ in its row must be incrementable (since the entry immediately below it in the SSYT, if any, will be at least as large as $k+1$).  
\end{remark}

In discussing any particular SSYT below, let $r_j$ be the $j$-th row in this SSYT  which 
contains a decrementable
copy of $k+1$. 

\begin{remark}\label{diamond-pairs}
Consider a row $r_j$ of a SSYT 
which is not the lowest such row. Let $i<k$ be the entry immediately above the leftmost copy of $k+1$ in row $r_{j+1}$. Then 
we may decrement the leftmost
$k+1$ in row $r_j$ and increment the rightmost $i$ in row 
$r_{j+1}-1$  at the same time to obtain a SSYT.  This is because $i+1\le k$, so that 
we preserve the property of rows weakly increasing from left to right even in the case that
$r_j = r_{j+1}-1$; preservation of column strictness is clear.
\end{remark}

We will use these pairs of covering relations colored $k$ and $i$ to provide  diamond relations
to be used in the algorithm below.

Now we are ready to complete the description of the algorithm.  The idea for handling color $k$
is to proceed from  top to bottom through each $k$-component. More precisely, we repeatedly choose a 
maximal element $S$ in the remaining part of a $k$-component and determine the edge products for all edges
originating at $S$ whose edge products have not yet been determined; these must necessarily 
all proceed downward from $S$.
At each stage, i.e. at each poset element $S$ encountered, we 
first force the edge products for 
all but one of its downward edges colored $k$;
essentially, we use the 
diamond condition for the diamond involving a  downward covering relation $R \prec S$
which is colored $k$ 
and an upward covering relation $S \prec U$ which is colored $i$ for some $i<k$, as follows.
Assume that the covering relation $R \prec S$ corresponds to decrementing the leftmost $k+1$ in a row $r_j$ of $S$ which is not the lowest such row. 
Now letting $i$ be
the entry immediately above the leftmost copy of $k+1$ in row $r_{j+1}$ of $S$, like in Remark \ref{diamond-pairs}, obtain $U$ from $S$ 
by incrementing the rightmost $i$ in row $r_{j+1}-1$.  Obtain the fourth poset element $T$
comprising the diamond by taking $U$ and decrementing the leftmost $k+1$ in row $r_j$, cf. Remark \ref{diamond-pairs}. 
Notice that three of the four edge products 
in this diamond will have already been determined, by virtue of either being  colored $i$ for
some $i<k$  or else 
involving a poset element $T'$ which is strictly greater than $U$ in the $k$-component of
the Gelfand-Tsetlin lattice; thus, (\ref{diam1}) forces the last edge product (cf. Remark ~\ref{diamond-product}).  
This leaves exactly 
one edge product yet to be determined among the downward edges colored $k$ which  originate at $S$; it is the product for the edge that corresponds to decrementing the rightmost $k+1$ in the lowest row $r_j$. 
But 
this edge product is determined by the crossing condition.

\begin{theorem}\label{sol-and-min}
The above algorithm shows explicitly how to determine all the edge products of
the Gelfand-Tsetlin lattice, and thus implies that this lattice has the solitary property.  It also has the edge minimality property.
\end{theorem}

\begin{proof}
Uniqueness of edge products, or in other words the solitary property, is 
proven above within the description of the algorithm.  The point is that we show how 
each edge product in turn is forced by earlier ones.
Edge minimality then follows from the 
fact that none of the resulting edge products are 0, 
cf. (\ref{pist}).    
\end{proof}

The fact that the Gelfand-Tsetlin lattice has the solitary and edge minimality properties was
previously proven in 
\cite[Theorem 4.4]{donepb}, but not in a constructive manner, so not in a way which provides
an algorithm to determine all of the edge products.  It would be interesting now to relate the
algorithm we have just given to the known edge labels.
%


\end{document}